\documentclass[11pt,reqno]{amsart}

\usepackage{amssymb,latexsym}
\usepackage{graphicx}
\usepackage{url}		

\calclayout
\makeatletter
\newcommand{\monthyear}[1]{%
  \def\@monthyear{\uppercase{#1}}}
\newcommand{\volnumber}[1]{%
  \def\@volnumber{\uppercase{#1}}}
\AtBeginDocument{%
\def\ps@plain{\ps@empty
  \def\@oddfoot{\@monthyear \hfil \thepage}%
  \def\@evenfoot{\thepage \hfil \@volnumber}}
\def\ps@firstpage{\ps@plain}
\def\ps@headings{\ps@empty
  \def\@evenhead{%
    \setTrue{runhead}%
    \def\thanks{\protect\thanks@warning}%
    \uppercase{The Fibonacci Quarterly}\hfil}%
  \def\@oddhead{%
    \setTrue{runhead}%
    \def\thanks{\protect\thanks@warning}%
    \hfill\uppercase{A new combinatorial interpretation of $F_n^2$}}%
  \let\@mkboth\markboth
  \def\@evenfoot{%
    \thepage \hfil \@volnumber}%
  \def\@oddfoot{%
    \@monthyear \hfil \thepage}%
  }%
\footskip=25pt
\pagestyle{headings}%
}
\makeatother

\theoremstyle{plain}
\numberwithin{equation}{section}
\newtheorem{thm}{Theorem}[section]

\newtheorem{lemma}[thm]{Lemma}

\newtheorem{idn}[thm]{Identity}
\newcommand{\floor}[1]{\lfloor#1\rfloor} 

\begin{document}
\bibliographystyle{fq} 

\monthyear{Month Year}
\volnumber{Volume, Number}
\setcounter{page}{1}

\title{A new combinatorial interpretation of the Fibonacci numbers squared}
\author{Kenneth Edwards}
\author{Michael A. Allen*} 
\address{Physics Department, Faculty of Science, 
Mahidol University, Rama 6 Road, Bangkok
10400 Thailand}

\email{kenneth.edw@mahidol.ac.th} 
\email{maa5652@gmail.com \textrm{(*corresponding author)}}

\begin{abstract}
We consider the tiling of an $n$-board (a $1\times n$ array of square
cells of unit width) with half-squares ($\frac12\times1$ tiles) and
$(\frac12,\frac12)$-fence tiles. A $(\frac12,\frac12)$-fence tile is
composed of two half-squares separated by a gap of width $\frac12$. We
show that the number of ways to tile an $n$-board using these 
types of tiles equals $F_{n+1}^2$ where $F_n$ is the $n$th Fibonacci
number. We use these tilings to devise combinatorial proofs of
identities relating the Fibonacci numbers squared to one another and
to other number sequences. Some of these identities appear to be new.
\end{abstract}

\maketitle

\section{Introduction}
The Fibonacci numbers can be interpreted as the number of ways to tile
an $n\times1$ array of joined $1\times1$ cells 
(called an $n$-board) with $1\times1$ squares
and $2\times1$ dominoes \cite{BCCG96,BQ=03}. More generally, the
number of ways to tile an $n$-board with all the $r\times1$
$r$-ominoes up to $r=k$ is the $k$-step (or $k$-generalized) Fibonacci
number
$F^{(k)}_{n+1}=F^{(k)}_{n}+F^{(k)}_{n-1}+\cdots+F^{(k)}_{n-k+1}$, with
$F^{(k)}_1=1$ and $F^{(k)}_{n<1}=0$ \cite{BQ=03}.  
In \cite{Edw08} it was shown that it is possible to obtain a
combinatorial interpretation of the Tribonacci numbers (the 3-step
Fibonacci numbers) as the number of tilings of an $n$-board using just
two types of tiles, namely, squares and $(\frac12,1)$-fence tiles.  A
$(w,g)$-fence tile is composed of two pieces (called posts) of size
$w\times1$ separated by a gap of size $g\times1$. In \cite{EA15} a
bijection between tiling with $(\frac12,g)$-fence tiles where
$g\in\{0,1,2,\ldots\}$ and strongly restricted permutations was
identified and then used to obtain results concerning the permutations
in a straightforward way.

Here we show that the number of ways to tile an $n$-board using
half-squares ($\frac12\times1$ tiles)
and $(\frac12,\frac12)$-fences is a Fibonacci number squared. The
Fibonacci numbers $F_n$ are given by
\begin{equation}\label{e:fib}
F_{n}=F_{n-1}+F_{n-2}, \quad F_0=0, \; F_1=1,\quad n\geq2.
\end{equation}  
We use these tilings to formulate combinatorial proofs in the style of
those found in \cite{BQ=03}
of a number of
identities relating to the Fibonacci numbers squared.

\section{Types of metatile}
When tiling a board with fences one must first determine the types of
metatile since any tiling of the board can be expressed as a tiling
using metatiles~\cite{Edw08}.  A metatile is a minimal arrangement of
tiles that exactly covers an integral number of adjacent
cells~\cite{Edw08,EA15}.  When tiling with half-squares (denoted by
$h$) and $(\frac12,\frac12)$-fence tiles (henceforth referred to simply as
fences and denoted by $f$), the simplest metatile is two half-squares
($h^2$) as illustrated in Fig.~\ref{f:metatiles}. 
This is the only metatile of length 1. The simplest metatile
of length 2 is the bifence ($f^2$) which is two interlocking fences.
Unlike the case of tiling with squares and $(\frac12,1)$-fences, the
filled fence, which is obtained by filling the gap in the fence by a
half-square and is denoted by $fh$, is not a metatile since it has
length $\frac32$. However, by adding a half-square to one end we form
length-2 metatiles ($hfh$ and $fhh$). We may also create metatiles of
length $2(j+1)$ by adding $j$ bifences between a half-square and
filled fence. Concatenating two filled fences generates a length-3
metatile ($fhfh$) and inserting $j$ bifences between them generates a
metatile of length $2j+3$. Inserting $j$ bifences between two
half-squares generates a metatile of length $2j+1$. Hence there are
two metatiles of length $l$ for all $l\geq3$.

\begin{figure} 
\begin{center}
\includegraphics[width=14cm]{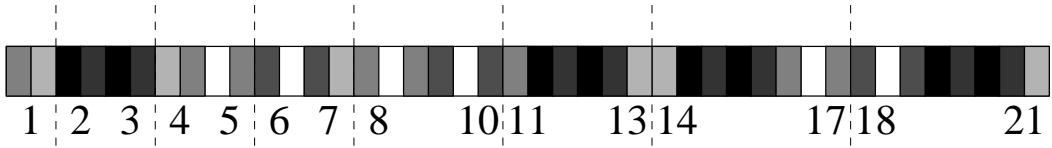}
\end{center}
\caption{A 21-board tiled with all possible metatiles of length less
  than 5: two half-squares $h^2$
  (cell 1), a bifence $f^2$ (cells 2--3), $hfh$ (cells 4--5), $fhh$
  (cells 6--7), $fhfh$ (cells 8--10), $hf^2h$ (cells 11--13), $hf^2fh$
  (cells 14--17), $fhf^2h$ (cells 18--21).}
\label{f:metatiles}
\end{figure}

We will refer to a metatile containing both $h$ and $f$ as being
\textit{mixed}. Notice that there are two types of mixed metatile of
each length larger than 1. This fact is used in the proofs of
Identities \ref{I:n+2} and \ref{I:2n+1}. An $h$ is said to be
\textit{captured} if it lies between the posts of a single fence, and
\textit{free} otherwise. These notions are used in the proof of
Identity~\ref{I:cass}. A bifence is said to be free if it
is not part of a larger metatile.

\section{The correspondence between numbers of tilings and the squares
of the Fibonacci numbers}

Let $A_n$ be the number of ways to tile an $n$-board using half-squares and
fences.
\begin{lemma}\label{L:Anidn}
\begin{equation}\label{e:Anidn}
A_{n}=\delta_{n,0}+A_{n-1}+3A_{n-2}+2\sum_{i=3}^nA_{n-i},
\end{equation} 
where $\delta_{i,j}$ is 1 if $i=j$ and 0 otherwise, and $A_{n}=0$ for $n<0$. 
\end{lemma}
\begin{proof}
Following \cite{BHS03,EA15}, we
condition on the last metatile. If the last metatile is of length $l$
there will be $A_{n-l}$ ways to tile the remaining $n-l$ cells.
There is one metatile of length 1, three of length 2, and two metatiles of
length $l$ for each $l\geq3$.
If $n=l$ there is exactly one tiling corresponding to that final
metatile so we make $A_0=1$. There is no way to tile an $n$-board if
$n<l$ and so $A_{n<0}=0$.
\end{proof}
Note that the inclusion of the initial value of the sequence via the
$\delta_{n,0}$ term is used to simplify the derivations of 
recurrence relations obtained from recurrence relations with sums
containing arbitrarily many terms by subtracting the recursion
relation for $A_{n-L}$ for some $L$.
Without it, one would have to consider the $n=L$ case
separately. 

\begin{thm}\label{T:bij}
The number of ways to tile an $n$-board using half-squares and
$(\frac12,\frac12)$-fences is $F_{n+1}^2$.
\end{thm}
\begin{proof}
Subtracting \eqref{e:Anidn} with $n$ changed to $n-1$ from
\eqref{e:Anidn} gives
\begin{equation}\label{e:An} 
A_n=\delta_{n,0}-\delta_{n,1}+2A_{n-1}+2A_{n-2}-A_{n-3},
\end{equation} 
i.e.,  $A_n=2A_{n-1}+2A_{n-2}-A_{n-3}$
for $n\ge3$, where $A_0=A_1=1$, $A_2=4$.
The Fibonacci numbers squared obey \eqref{e:An} with
$A_n=F_{n+1}^2$ (see Identity~30 in \cite{BQ=03}).  
\end{proof}

\section{Identities}

\begin{idn}\label{I:sum} 
\begin{equation}\label{e:sumidn}
F_{n}^2=F_{n-1}^2+3F_{n-2}^2+2\sum_{i=3}^nF_{n-i}^2, \quad n\ge2.
\end{equation} 
\end{idn}
\begin{proof}
It follows from Lemma~\ref{L:Anidn} and Theorem~\ref{T:bij}. 
\end{proof}

\begin{idn}\label{I:n+2}
\begin{equation}\label{e:n+2} 
F_{n+3}^2-1=\sum_{k=0}^n\left\{3F_{k+1}^2+2\sum_{i=1}^kF_i^2\right\},
\quad n\ge0.
\end{equation} 
\end{idn}
\begin{proof}
How many ways are there to tile an $(n+2)$-board using at least 1 fence? Answer
1: $A_{n+2}-1$ since this corresponds to all tilings except the
all-$h$ tiling.  Answer 2: condition on the location of the last
fence. Suppose this fence lies on cells $k+1$ and $k+2$
($k=0,\ldots,n$).  Either there is a bifence covering these cells and
so there are $A_k$ ways to tile the remaining cells, or the cells are
at the end of a mixed metatile and so there are
$2(A_{k+2-2}+A_{k+2-3}+\cdots+A_0)$ ways to tile the remaining
cells. Hence, equating the two answers,
\[
A_{n+2}-1=\sum_{k=0}^n\left\{3A_k+2(A_0+A_1+\cdots+A_{k-1})\right\}.
\]
The identity then follows from Theorem~\ref{T:bij}.
\end{proof}

\begin{idn}\label{I:2n+1}
\begin{equation}\label{e:2n+1} 
F_{2n+2}^2=F_1^2+\sum_{k=1}^n\left\{F_{2k+1}^2+2\sum_{i=1}^{2k}F_i^2\right\},
\quad n\ge0.
\end{equation} 
\end{idn}
\begin{proof}
How many ways are there to tile an $(2n+1)$-board? Answer 1: $A_{2n+1}$.  Answer
2: an odd-length board must have at least one $h$ and the final $h$
must be on an odd cell. Condition on the location of the last
$h$. Suppose that the last $h$ is in cell $2k+1$
($k=0,\ldots,n$). Either it is part of $h^2$ and so there are $A_{2k}$
ways to tile the remaining cells, or it is part of a mixed metatile
and so there are $2(A_{2k+1-2}+A_{2k+1-3}+\cdots+A_0)$ ways to tile
the remaining cells. In the latter case, evidently, $k$ cannot be
zero. Hence
\[
A_{2n+1}=\sum_{k=0}^nA_{2k}+2\sum_{k=1}^n(A_0+A_1+\cdots+A_{2k-1}).
\] 
The identity then follows from Theorem~\ref{T:bij}.
\end{proof}

\begin{lemma}\label{L:Sn}
If $S_n$ is the number of ways to tile an $n$-board using half-squares and
$(\frac12,\frac12)$-fences such that no free bifences occur then
\begin{equation}\label{e:Sn}
S_n=2S_{n-1}+S_{n-2}, \quad n\ge2,
\end{equation} 
where $S_0=S_1=1$.
\end{lemma}
\begin{proof}
Conditioning on the last metatile (which can be any metatile except a
bifence) gives
\begin{equation}\label{e:Snidn} 
S_n=\delta_{0,n}+S_{n-1}+2\sum_{j=2}^n S_{n-j}
\end{equation}
Subtracting \eqref{e:Snidn} with $n$ replaced by
$n-1$ from \eqref{e:Snidn} gives 
$S_n=\delta_{0,n}-\delta_{1,n}+2S_{n-1}+S_{n-2}$ which is equivalent to
\eqref{e:Sn}. 
\end{proof}

$S_{n\ge0}=1,1,3,7,17,41,99,239,577,1393,\ldots$ is sequence
A001333 in \cite{Slo-OEIS}.

\begin{idn}\label{I:nS} 
\begin{equation}\label{e:nS}
F_{n+1}^2=S_n+\sum_{k=2}^nF^2_{k-1}S_{n-k}
\end{equation}  
\end{idn}
\begin{proof}
How many tilings of an $n$-board contain at least one free bifence?
Answer 1: $A_n-S_n$. Answer 2: condition on the location of the last
free bifence.  The number of tilings when this final free bifence lies
on cells $k-1$ and $k$ (for $k=2,\ldots,n$) is
$A_{k-2}S_{n-k}$. Summing over all possible $k$ and equating the two
answers gives
\[
A_n-S_n=\sum_{k=2}^nA_{k-2}S_{n-k}
\]
and the identity follows from Theorem~\ref{T:bij}.
\end{proof}

\begin{lemma}\label{L:Cn}
If $C_n$ is the number of ways to tile an $n$-board using half-squares and
$(\frac12,\frac12)$-fences such that no bifences occur then
\begin{equation}\label{e:Cn}
C_n=C_{n-1}+2C_{n-2}+C_{n-3}, \quad n\ge3,
\end{equation} 
where $C_0=C_1=1$ and $C_2=3$.
\end{lemma}
\begin{proof}
Conditioning on the last metatile which can only be $h^2$, $fhh$, $hfh$,
or $fhfh$ gives $C_n=\delta_{0,n}+C_{n-1}+2C_{n-2}+C_{n-3}$ which is
equivalent to \eqref{e:Cn}.
\end{proof}

$C_{n\ge0}=1,1,3,6,13,28,60,129,277,595,\ldots$ is sequence
A002478 in \cite{Slo-OEIS} and consists of the even terms of
Narayana's cows sequence (A000930 in \cite{Slo-OEIS}).

\begin{idn}\label{I:nC} 
\begin{equation}\label{e:nC}
F_{n+1}^2=C_n+\sum_{k=2}^nF^2_{k-1}C_{n-k}
+\sum_{k=3}^n\sum_{l=3}^k(2-\delta_{l,3})F^2_{k-l+1}C_{n-k}.
\end{equation}  
\end{idn}
\begin{proof}
How many tilings of an $n$-board contain at least one bifence?  Answer
1: $A_n-C_n$. Answer 2: condition on the location of the last metatile
containing a bifence. If the metatile is just a bifence, it can occupy
cells $k-1$ and $k$ ($k=2,\ldots,n$). The number of tilings in
this case is $A_{k-2}C_{n-k}$. Otherwise the metatile can occupy cells
$k-l+1$ to $k$ ($k=l,\ldots,n$) 
where $l$ is the length of the metatile. The number of tilings for a
metatile of length $l$ for a given $k$ is $A_{k-l}C_{n-k}$.
There is one $l=3$
metatile containing a bifence and two for each $l>3$. Summing over the
possible $l$ and $k$ and equating the two answers gives
\[
A_n-C_n=\sum_{k=2}^nA_{k-2}C_{n-k}
+\sum_{k=3}^n\sum_{l=3}^k(2-\delta_{l,3})A_{k-l}C_{n-k}.
\]
The identity then follows from Theorem~\ref{T:bij}.
\end{proof}

\begin{lemma}\label{L:Tn}
If $T_n$ is the number of ways to tile an $n$-board using half-squares and
$(\frac12,\frac12)$-fences such that no even-length metatiles occur then
\begin{equation}\label{e:Tn}
T_n=T_{n-1}+T_{n-2}+T_{n-3}, \quad n\ge3, 
\end{equation} 
where $T_0=T_1=T_2=1$.
\end{lemma}
\begin{proof}
Condition on the last metatile of which there is only one of length 1 but
two of every other odd length. Thus
\begin{equation}\label{e:Tnidn} 
T_n=\delta_{0,n}+T_{n-1}
+2\!\!\!\!\!\sum_{j=1}^{\floor{(n-1)/2}}\!\!\!\!\!T_{n-1-2j}
\end{equation}
with $T_{n<0}=0$. Subtracting \eqref{e:Tnidn} with $n$ replaced by
$n-2$ from \eqref{e:Tnidn} gives 
$T_n=\delta_{0,n}-\delta_{2,n}+T_{n-1}+T_{n-2}+T_{n-3}$
which is equivalent to
\eqref{e:Tn}. 
\end{proof}

$T_{n\ge0}=1,1,1,3,5,9,17,31,57,105,\ldots$ is a Tribonacci sequence
(A000213 in \cite{Slo-OEIS}).

\begin{idn}\label{I:nT} 
\begin{equation}\label{e:nT}
F_{n+1}^2=T_n
+\sum_{k=2}^n\sum_{j=1}^{\floor{k/2}}(2+\delta_{j,1})F^2_{k-2j+1}T_{n-k}.
\end{equation}  
\end{idn}
\begin{proof}
How many tilings of an $n$-board contain at least one even-length
metatile?  
Answer 1: $A_n-T_n$. Answer 2: condition on the location of the last
even-length metatile. The metatile can occupy cells
$k-2j+1$ to $k$ ($k=2j,\ldots,n$) 
where $2j$ is the length of the metatile. The number of tilings for a
metatile of length $2j$ for a given $k$ is $A_{k-2j}T_{n-k}$.
There are three metatiles of length 2 and two for each $j>1$.
Summing over the
possible $j$ and $k$ and equating the two answers gives
\[
A_n-T_n=\sum_{k=2}^n\sum_{j=1}^{\floor{k/2}}(2+\delta_{j,1})A_{k-2j}T_{n-k}.
\]
The identity then follows from Theorem~\ref{T:bij}.
\end{proof}

\begin{idn}\label{I:cass} 
\begin{equation}\label{e:cass}
F_{n+1}^2=3F_n^2-F^2_{n-1}+2(-1)^n.
\end{equation}  
\end{idn}
\begin{proof}
We find a near bijection between the tilings of an $n$-board and an
$(n-2)$-board and the tilings of three $(n-1)$-boards. There is an
exact bijection between the first $(n-1)$-board and tilings of an
$n$-board that end in $h^2$. There is a bijection $\mathcal{B}_n$
between the tilings of the $n$-board that end in a fence and the
tilings of the second $(n-1)$-board if both boards contain at least
one $h$ (i.e., neither is an all-bifence tiling). If the $n$-board
ends in a bifence, find the final $h$. If the final $h$ is captured,
replace the filled fence it is in by $h$. Otherwise replace the free
$h$ and the bifence to the right of it by a filled fence. This gives
all $h$-containing $(n-1)$-board tilings ending with a fence. The
tilings ending in $h$ are obtained from the $n$-board tilings ending
in a filled fence by replacing that filled fence by an $h$. This
leaves the tilings of the $n$-board that end in a free $h$ (but not
$h^2$). Not counting this final $h$, find the final $h$ in the tiling
(as there must be at least one other $h$) and then obtain the
corresponding $(n-1)$-board by using the same procedure as for
$n$-boards ending in a fence. This generates all tilings of the third
$(n-1)$-board ending in a free $h$. The bijection between the
remaining $h$-containing tilings of this board (i.e., those ending in
a fence) and the all $h$-containing tilings of an $(n-2)$-board is
simply $\mathcal{B}_{n-1}$. When $n$ is even, the $n$ and
$(n-2)$-boards both have an all bifence tiling and so 
$A_{n}+A_{n-2}=3A_{n-1}+2$. When $n$ is odd, the second and third
$(n-1)$-boards have all bifence tilings which do not correspond to any
of the $h$-containing tilings of the $n$ or $(n-2)$-boards and so we must
subtract 2. Thus overall,
\[
A_n+A_{n-2}=3A_{n-1}+2(-1)^n,
\]
and the required identity is obtained from using Theorem~\ref{T:bij}. 
\end{proof}

\section{Discussion}
Clearly, the tiling given here has the same combinatorial interpretation
as tiling an even-length board
with squares and $(1,1)$-fences. However, in that expanded form of
tiling, there are only three metatiles (the square, filled fence, and
bifence) and so the identities we obtain here would not arise so
naturally. The number of tilings of an odd-length board in this case
gives the golden rectangle numbers and so identities relating these to
the Fibonacci numbers squared can also be obtained via combinatorial proof. 
 
It should be noted that the occurrence of an infinite number of
metatiles is not limited to tiling with fences. For example, if one
tiles an $n$-board with half-squares and squares ($S$) then the
metatiles are $hS^jh$ for $j\ge0$ (and the number of such tilings is
$F_{2n+1}$ since it is equivalent to tiling an even length board with
squares and dominoes).


\medskip

\noindent MSC2010:  05A19, 11B39

\end{document}